\documentclass[a4paper, twoside, reqno]{amsart} 
\usepackage{xcolor}
\usepackage{amssymb} 
\usepackage{amsthm}
\usepackage{amsmath}
\usepackage{mathtools} 

\usepackage{microtype}

\usepackage{hyperref} 
\definecolor{bluey}{rgb}{0,0,0.6}
\hypersetup{colorlinks,
linkcolor=bluey,
citecolor=bluey,
urlcolor=bluey}

\usepackage{mmacells} 
\usepackage{changepage} 

\newcommand{\CC}{\mathbf C}

\newmuskip\pFqmuskip
\newcommand*\pFq[6][8]{%
  \begingroup 
  \pFqmuskip=#1mu\relax
  \mathchardef\normalcomma=\mathcode`,
  \mathcode`\,=\string"8000
  \begingroup\lccode`\~=`\,
  \lowercase{\endgroup\let~}\pFqcomma
  {}_{#2}F_{#3}{\left(\genfrac..{0pt}{}{#4}{#5};#6\right)}%
  \endgroup
}
\newcommand{\pFqcomma}{{\normalcomma}\mskip\pFqmuskip}

\makeatletter
\def\thmhead@plain#1#2#3{%
  \thmname{#1}\thmnumber{\@ifnotempty{#1}{ }\@upn{#2}}%
  \thmnote{ \the\thm@notefont(#3)}}
\let\thmhead\thmhead@plain
\makeatother

\newtheorem{theorem}{Theorem}
\newtheorem{appendixtheorem}{Theorem}
\newtheorem{lemma}{Lemma} 

\begin{document}

\title[Two product formulas for counting successive vertex orderings]%
{Two product formulas for counting\\successive vertex orderings}
\author[B.~S. Ho]{Boon Suan Ho}
\address{College of Design and Engineering, National University of Singapore}
\email{\href{mailto:hbs@u.nus.edu}{hbs@u.nus.edu}}

\begin{abstract} 
A \emph{vertex ordering} of a graph $G$ is a 
bijection $\pi\colon\{1,\dots,|V(G)|\}\to V(G)$. 
It is \emph{successive} if the induced subgraph 
$G[v_{\pi(1)},\dots,v_{\pi(k)}]$ is connected for each $k$. 
Fang et al.\ [{\sl J. Comb.\ Theory} \textbf{A199} (2023), 105776]
gave formulas for counting the number of successive vertex orderings
for a class of graphs they called \emph{fully regular}, and conjectured
that these formulas could be written as certain products involving 
differences or ratios of binomial coefficients in two cases: When the
graph is the line graph $L(K_n^{(3)})$ of the complete $3$-uniform hypergraph, 
or when it is the line graph $L(K_{m,n}^{(1,2)})$ of a complete ``bipartite'' $3$-uniform hypergraph. 
In this paper, we confirm both of these conjectures.
\end{abstract}

\maketitle

A \emph{vertex ordering} of a graph $G$ is a 
bijection $\pi\colon\{1,\dots,|V(G)|\}\to V(G)$. 
It is \emph{successive} if the induced subgraph 
$G[v_{\pi(1)},\dots,v_{\pi(k)}]$ is connected for each $k$. 
Fang et al.\ \cite{fang} defined the class of \emph{fully regular graphs},
and proved that for such graphs, the probability of a random vertex ordering being successive
is given by a sum of products involving certain invariants of the graph. In two special cases, 
they conjectured that this sum admits a nice product expression. 
By manipulating certain generalized hypergeometric identities, 
we prove those conjectures in this paper.

\begin{theorem}[Fang et al.\ \cite{fang}, Conjecture 5.2] \label{theorem:1}
We have, for $n\ge3$,
\[\sum_{i=0}^{\lfloor n/3\rfloor-1}\prod_{j=1}^i\frac{-\binom{n-3j}{3}}{\binom{n}{3}-\binom{n-3j}{3}}
=\Bigl\lfloor\frac{n}{3}\Bigr\rfloor\prod_{\substack{j=n+1\\3\nmid j}}^{\lfloor3n/2\rfloor-2}c_j\Bigg/\prod_{\substack{j=3\\3\mid j}}^{n-3}c_j,\]
where $c_j=\frac{6}{j}(\binom{n}{3}-\binom{n-j}{3})=j^2 + (3-3n)j + 3n^2-6n+2$.
\end{theorem}

By Theorem 1.4 of \cite{fang}, the expression on the left is equal to the probability that a vertex ordering of $L(K_n^{(3)})$
is successive; equivalently, the expression multiplied by $|V|!$ counts the number of successive vertex orderings of $L(K_n^{(3)})$.
(We write $L(G)$ for the \emph{line graph} of $G$, which is a graph 
whose vertices are the edges of $G$, where two edges are considered adjacent
if they share a common vertex.)
\begin{theorem}[Fang et al.\ \cite{fang}, Conjecture 5.3] \label{theorem:2}
We have, for $m\ge1$ and $n\ge2$,
\[
\sum_{i=0}^{\min\{m,\lfloor n/2\rfloor\}}\prod_{j=1}^i\frac{-(m-j)\binom{n-2j}{2}}{m\binom{n}{2}-(m-j)\binom{n-2j}{2}}
=m\prod_{j=1}^{m-1}\frac{mn-\binom{m+1}{2}+\binom{j}{2}}{d_j},
\]
where $d_j=\frac{1}{j}(m\binom{n}{2}-(m-j)\binom{n-2j}{2})$,
and where any numerators or denominators on the right that are equal to zero are ignored.
\end{theorem}

Similarly, by Theorem 1.4 of \cite{fang}, the expression on the left gives the probability of 
getting a successive vertex ordering of $L(K_{m,n}^{(1,2)})$, where $K_{m,n}^{(1,2)}$
is the $3$-uniform hypergraph on vertex set $V=V_1\cup V_2$ with $|V_1|=m$
and $|V_2|=n$, whose edges are all of the subsets $e\subseteq V$ such that
$|e\cap V_1|=1$ and $|e\cap V_2|=2$.

Recall that the \emph{generalized hypergeometric function} is defined by
\[\pFq{p}{q}{a_1,\dots,a_p}{b_1,\dots,b_q}{z}\coloneqq\sum_{k=0}^\infty\frac{(a_1)_k\dots(a_p)_k}{(b_1)_k\dots(b_q)_k}\frac{z^k}{k!},\]
where $(a)_k\coloneqq a(a+1)\dots(a+k-1)$ denotes the \emph{rising factorial} (with $(a)_0\coloneqq1$).
Notice that if one of the upper parameters is a nonpositive integer $-a$, then this sum has only $a+1$ terms,
since the rest of the terms vanish (we call such hypergeometric sums \emph{terminating}). 
In particular, if $b>a\ge0$ are integers and we have a ${}_3F_2(1)$ sum where one of the upper parameters 
is $-a$ and one of the lower parameters is $-b$, this sum is well-defined.
We will only be interested in the terminating ${}_3F_2(1)$ case.

\section{Proof of Theorem 1}\label{sec:theorem1}
Let $X_n\coloneqq\sqrt{1+6n-3n^2}$. Then
\[\begin{cases}
-\binom{n-3j}{3}&=-\frac{1}{6}(n-3j)(n-1-3j)(n-2-3j),\\
\binom{n}{3}-\binom{n-3j}{3}&=\frac{9}{2}j\bigl(j-\frac{1}{6}(3n-3+X_n)\bigr)\bigl(j-\frac{1}{6}(3n-3-X_n)\bigr);
\end{cases}\]
and so we have
\begin{align*}
&\sum_{i=0}^{\lfloor n/3\rfloor-1}\prod_{j=1}^i\frac{-\binom{n-3j}{3}}{\binom{n}{3}-\binom{n-3j}{3}} \\
&\quad=\sum_{i=0}^{\lfloor n/3\rfloor-1}\prod_{j=1}^i
\frac{-\frac{1}{6}(n-3j)(n-1-3j)(n-2-3j)}
{\frac{9}{2}j\bigl(j-\frac{1}{6}(3n-3+X_n)\bigr)\bigl(j-\frac{1}{6}(3n-3-X_n)\bigr)}\\
&\quad=\sum_{i=0}^{\lfloor n/3\rfloor-1}\prod_{j=1}^i
\frac{(-\frac{n}{3}+j)(\frac{1}{3}-\frac{n}{3}+j)(\frac{2}{3}-\frac{n}{3}+j)}
{j\bigl(j-\frac{1}{6}(3n-3+X_n)\bigr)\bigl(j-\frac{1}{6}(3n-3-X_n)\bigr)}\\
&\quad=\pFq{3}{2}{1-\frac{n}{3}, \frac{4}{3}-\frac{n}{3}, \frac{5}{3}-\frac{n}{3}}{\frac{1}{6}(9-3n+X_n), \frac{1}{6}(9-3n-X_n)}{1}
\eqqcolon F(1).
\end{align*}
Sheppard (Equation (18) in \cite{sheppard}; see Corollary 3.3.4 in \cite{specialfunctions} for a textbook treatment,
as well as \hyperref[appendixtheorem:A1]{Theorem A1} in \hyperref[sec:appendix]{Appendix A} for a self-contained proof) 
proved the identity
\[\pFq{3}{2}{-N,a,b}{d,e}{1}
=\frac{(d-a)_N(e-a)_N}{(d)_N(e)_N}
\pFq{3}{2}{-N,a,a+b-N-d-e+1}{a-N-d+1,a-N-e+1}{1},\]
whenever $N\ge0$ is a nonnegative integer and $a,b,d,e\in\CC$
are such that both sides converge.
In our sum $F(1)$, one of the upper parameters is a nonpositive integer, though exactly which parameter that is 
depends on the value of $n\bmod 3$. \emph{In the rest of this section, we will assume that $n\bmod3=0$},
as the other two cases are proven analogously (see \hyperref[sec:appendix]{Appendix A} for those cases). 
Set $N=n/3-1$, $a=4/3-n/3$, $b=5/3-n/3$,
$d=\frac{1}{6}(9-3n+X_n)$, and $e=\frac{1}{6}(9-3n-X_n)$ in
Sheppard's identity to get
\begin{align*}
F(1)&=\frac{\bigl(\frac{1}{6}(1-n+X_n)\bigr)_{n/3-1}
\bigl(\frac{1}{6}(1-n-X_n)\bigr)_{n/3-1}}
{\bigl(\frac{1}{6}(9-3n+X_n)\bigr)_{n/3-1}
\bigl(\frac{1}{6}(9-3n-X_n)\bigr)_{n/3-1}}\\
&\qquad\times\pFq{3}{2}{1-\frac{n}{3},\frac{4}{3}-\frac{n}{3},2}{\frac{1}{6}(11-n+X_n),\frac{1}{6}(11-n-X_n)}{1}.
\end{align*}
The denominators and lower parameters are never nonpositive integers,
since $X_n$ contributes a nonzero imaginary part, so we may proceed without worrying about division by zero. 
Now $(a)_k(b)_k=\prod_{j=0}^{k-1}(a+j)(b+j)$, so we may simplify
\begin{align*}
&\Bigl(\frac{1}{6}\bigl(1-n+X_n\bigr)\Bigr)_{n/3-1}
\Bigl(\frac{1}{6}\bigl(1-n-X_n\bigr)\Bigr)_{n/3-1} \\
&\quad=\prod_{j=0}^{n/3-2}\Bigl(j^2-\frac{n-1}{3}j+\frac{n^2-2n}{9}\Bigr)
\end{align*}
and
\begin{align*}
&\Bigl(\frac{1}{6}\bigl(9-3n+X_n\bigr)\Bigr)_{n/3-1}
\Bigl(\frac{1}{6}\bigl(9-3n-X_n\bigr)\Bigr)_{n/3-1} \\
&\quad=\prod_{j=0}^{n/3-2}\Bigl(j^2+(3-n)j+\frac{3n^2-15n+20}{9}\Bigr).
\end{align*}
As for the ${}_3F_2(1)$ sum on the right-hand side of Sheppard's identity,
we can use Gosper's algorithm to prove that
\[\pFq{3}{2}{1-\frac{n}{3},\frac{4}{3}-\frac{n}{3},2}{\frac{1}{6}(11-n+X_n),\frac{1}{6}(11-n-X_n)}{1}
=\frac{n^2-4n+6}{3n-6}.\] 
(See \hyperref[sec:gosper]{Section 3} for details.)
Putting it all together, our task is now to prove that
\[\frac{\prod_{j=0}^{n/3-2}(9j^2-3(n-1)j+n^2-2n)}{\prod_{j=0}^{n/3-2}(9j^2+9(3-n)j+3n^2-15n+20)}
\frac{n^2-4n+6}{3n-6}
=\frac{n/3}{\prod_{j=1}^{n/3-1}c_{3j}}
\frac{\prod_{j=n+1}^{\lfloor3n/2\rfloor-2}c_j}
{\prod_{j=n/3+1}^{\lfloor n/2\rfloor-1}c_{3j}}.\]
Since $c_{3j+3}=9j^2+9(3-n)j+3n^2-15n+20$,
the bottom left product is equal to the product $\prod_{j=1}^{n/3-1}c_{3j}$ on the bottom right. 
Thus it suffices for us to prove that
\begin{align*}
\prod_{j=0}^{n/3-2}\bigl(9j^2-3(n-1)j+n^2-2n\bigr)
\prod_{j=n/3+1}^{\lfloor n/2\rfloor-1}c_{3j}
=\frac{n(n-2)}{n^2-4n+6}\prod_{j=n+1}^{\lfloor3n/2\rfloor-2}c_j.
\end{align*}
Denote by $b_j\coloneqq9j^2-3(n-1)j+n^2-2n$ the $j$-th factor of the product on the left. 
Since $b_0/c_{n+1}=n(n-2)/(n^2-4n+6)$, we are left to prove that
\[\prod_{j=1}^{\lceil n/6\rceil-1}b_j\prod_{j=\lceil n/6\rceil}^{n/3-2}b_j\prod_{j=n/3+1}^{\lfloor n/2\rfloor-1}c_{3j}
=\prod_{j=n+2}^{\lfloor3n/2\rfloor-2}c_j.\]
Since $c_k=c_{3n-3-k}$, the result follows by using the fact that 
$b_j=c_{n-1+3j}$ in the first product and $b_j=c_{2n-2-3j}$ in the second product.
The idea is that the product on the right splits into the three products on the left, 
where each product on the left comprises only terms with the same index value modulo $3$.

\section{Proof of Theorem 2}\label{sec:theorem2}
The main ideas in this proof are very similar to those in the proof of \hyperref[theorem:1]{Theorem 1}, though the details get rather messy 
near the end. \emph{In this section, we will assume that $n$ is even}; the
reader is referred to \hyperref[sec:appendix]{Appendix A} for the (similar) odd case.
Let $Y_{m,n}\coloneqq\sqrt{(2m+1)^2-8mn}$. 
Since $-(m-j)\binom{n-2j}{2}=2(j-\frac{n}{2})(j+\frac{1}{2}-\frac{n}{2})(j-m)$
and $m\binom{n}{2}-(m-j)\binom{n-2j}{2}
=2j\bigl(j+\frac{1}{4}(1-2m-2n+Y_{m,n})\bigr)\bigl(j+\frac{1}{4}(1-2m-2n-Y_{m,n})\bigr)$,
we have
\begin{align*}
&\sum_{i=0}^{\min\{m,\lfloor n/2\rfloor\}}\prod_{j=1}^i\frac{-(m-j)\binom{n-2j}{2}}{m\binom{n}{2}-(m-j)\binom{n-2j}{2}} \\
&\quad=\pFq{3}{2}{1-\frac{n}{2},\frac{3}{2}-\frac{n}{2},1-m}
{\frac{1}{4}(5-2m-2n+Y_{m\normalcomma n}), 
\frac{1}{4}(5-2m-2n-Y_{m\normalcomma n})}{1}\\
&\quad=\frac{\bigl(\frac{1}{4}(-1-2m+Y_{m,n})\bigr)_{n/2-1}
\bigl(\frac{1}{4}(-1-2m-Y_{m,n})\bigr)_{n/2-1}}
{\bigl(\frac{1}{4}(5-2m-2n+Y_{m,n})\bigr)_{n/2-1}
\bigl(\frac{1}{4}(5-2m-2n-Y_{m,n})\bigr)_{n/2-1}}\\
&\qquad\times\pFq{3}{2}{1-\frac{n}{2},\frac{3}{2}-\frac{n}{2},2}{\frac{1}{4}(9+2m-2n+Y_{m\normalcomma n}), \frac{1}{4}(9+2m-2n-Y_{m\normalcomma n})}{1},
\end{align*}
where we have applied Sheppard's identity with $N=n/2-1$, $a=3/2-n/2$, $b=1-m$, 
$d=\frac{1}{4}(5-2m-2n+Y_{m,n})$, and $e=\frac{1}{4}(5-2m-2n-Y_{m,n})$.
Now the lower parameters of the resulting ${}_3F_2(1)$ sum are never nonpositive integers,
and while the numerators and denominators of the fractions involved can have factors that 
are negative integers, they never include a zero factor
(see \hyperref[lemma:A2]{Lemmas A2} and \hyperref[lemma:A3]{A3} 
in \hyperref[sec:appendix]{Appendix A} for proofs of these facts).
Thus we need not worry about division by zero in this case either.

The expression above simplifies to
\[\frac{\prod_{j=0}^{n/2-2}\bigl(j^2-\frac{2m+1}{2}j+\frac{mn}{2}\bigr)}
{\prod_{j=0}^{n/2-2}\bigl(j^2+\frac{5-2m-2n}{2}j+\frac{6-6m-5n+4mn+n^2}{4}\bigr)}
\frac{6+4m-5n+n^2}{4m}.
\]
(Once again, see \hyperref[sec:gosper]{Section 3} for details on evaluating this ${}_3F_2(1)$ sum
using Gosper's algorithm.)
We must prove that this expression is equal to 
\[m\prod_{j=1}^{m-1}\frac{mn-\binom{m+1}{2}+\binom{j}{2}}{d_j},\]
where $d_j\coloneqq\frac{1}{j}(m\binom{n}{2}-(m-j)\binom{n-2j}{2})$.
Since $d_{j+1}=2j^2+(5-2m-2n)j+\frac{1}{2}(6-6m-5n+4mn+n^2)$, our task is to prove that
\begin{equation}\label{eq:sus}
\frac{\prod_{j=0}^{n/2-2}p_j}{\prod_{j=0}^{n/2-2}d_{j+1}}
=\frac{4m^2}{6+4m-5n+n^2}\frac{\prod_{j=1}^{m-1}q_j}{\prod_{j=1}^{m-1}d_j},
\end{equation}
where we have defined $p_j\coloneqq 2j^2-(2m+1)j+mn$
and $q_j\coloneqq mn-\binom{m+1}{2}+\binom{j}{2}$.
\emph{The products in this section are to be interpreted as ignoring zero factors.}
(Formally, we could define such a product by viewing $\prod_j a_j$
as shorthand for $\prod_j (a_j+[a_j=0])$,
where $[P]$ denotes the \emph{Iverson bracket}, which is equal to $1$ if the proposition
$P$ is true, and is equal to $0$ if $P$ is false.)
Note that in \eqref{eq:sus}, only the two products on the right may contain zero factors, 
as mentioned earlier in our discussion concerning division by zero.

Since we have the identities
\begin{equation} \label{eq:refl}
q_j=q_{1-j},
\end{equation}
\begin{equation*}
p_j=q_{m-2j+1}=q_{2j-m},
\end{equation*}
\begin{equation*}
d_j=p_{m+n/2-j}=q_{m+n-2j},
\end{equation*}
and
\begin{equation*}
\frac{4m^2}{6+4m-5n+n^2}
=\frac{q_mq_{m+1}}{q_{m-n+1}q_{m-n+3}},
\end{equation*}
which are easily checked by direct calculation, we are left to prove the following:
\begin{equation}\label{eq:thm2identity}
\prod_{j=0}^{n/2}q_{m-2j+1}
\prod_{j=1}^{m-1}q_{m+n-2j}
=\prod_{j=1}^{m+1}q_j
\prod_{j=1}^{n/2-1}q_{m+n-2j}.
\end{equation}
To this end, we consider three cases, depending on the relative ordering of $n/2$, $m$, and $n$.

\textbf{Case 1} ($n/2<m$). We may divide both sides of \eqref{eq:thm2identity}
by $\prod_{j=1}^{n/2-1}q_{m+n-2j}$, so our task is to prove that
\begin{equation} \label{eq:thm2case1}
\prod_{j=0}^{n/2}q_{m-2j+1}
\prod_{j=n/2}^{m-1}q_{m+n-2j}
=\prod_{j=1}^{m+1}q_j.
\end{equation} 

\textbf{Case 1A} ($n\le m$). We simplify the second product on the left of \eqref{eq:thm2case1},
using \eqref{eq:refl}:
\begin{align*}
\prod_{j=n/2}^{m-1}q_{m+n-2j}
&=\prod_{j=n/2}^{n/2+\lceil m/2\rceil-1}q_{m+n-2j}
\prod_{j=n/2+\lceil m/2\rceil}^{m-1}q_{m+n-2j}\\
&=\prod_{j=0}^{\lceil m/2\rceil-1}q_{m-2j}
\prod_{j=n/2+\lceil m/2\rceil}^{m-1}q_{1-m-n+2j}\\
&=\prod_{j=0}^{\lfloor m/2\rfloor-n/2-1}q_{m-n-1-2j}
\prod_{j=0}^{\lceil m/2\rceil-1}q_{m-2j}.
\end{align*}
The left-hand side of \eqref{eq:thm2case1} then becomes
\begin{align*}
&\prod_{j=0}^{n/2}q_{m-2j+1}
\prod_{j=0}^{\lfloor m/2\rfloor-n/2-1}q_{m-n-1-2j}\prod_{j=0}^{\lceil m/2\rceil-1}q_{m-2j}\\
&\quad=\prod_{j=0}^{\lfloor m/2\rfloor}q_{m-2j+1}\prod_{j=0}^{\lceil m/2\rceil-1}q_{m-2j}\\
&\quad=\prod_{j=1}^{m+1}q_j,
\end{align*}
which completes the proof of case 1A.

\textbf{Case 1B} ($n>m$). We simplify the first product on the left of \eqref{eq:thm2case1},
using \eqref{eq:refl}:
\begin{align*}
\prod_{j=0}^{n/2}q_{m-2j+1}
&=\prod_{j=0}^{\lfloor m/2\rfloor}q_{m-2j+1}
\prod_{j=\lfloor m/2\rfloor+1}^{n/2}q_{m-2j+1}\\
&=\prod_{j=0}^{\lfloor m/2\rfloor}q_{m-2j+1}
\prod_{j=\lfloor m/2\rfloor+1}^{n/2}q_{2j-m}\\
&=\prod_{j=0}^{\lfloor m/2\rfloor}q_{m-2j+1}
\prod_{j=0}^{n/2-\lfloor m/2\rfloor-1}q_{n-m-2j}.
\end{align*}
The left-hand side of \eqref{eq:thm2case1} then becomes
\begin{align*}
&\prod_{j=0}^{\lfloor m/2\rfloor}q_{m-2j+1}
\prod_{j=n/2}^{m-1}q_{m+n-2j}
\prod_{j=0}^{n/2-\lfloor m/2\rfloor-1}q_{n-m-2j}\\
&\quad=\prod_{j=0}^{\lfloor m/2\rfloor}q_{m-2j+1}
\prod_{j=0}^{\lceil m/2\rceil-1} q_{m-2j}\\
&\quad=\prod_{j=1}^{m+1}q_j.
\end{align*}
This completes the proof of case 1B, and consequently case 1.

\textbf{Case 2} ($n/2\ge m$). We may divide both sides by 
$\prod_{j=1}^{m-1}q_{m+n-2j}$, so our task is to prove that
\[\prod_{j=0}^{n/2}q_{m-2j+1}
=\prod_{j=1}^{m+1}q_j\prod_{j=m}^{n/2-1}q_{m+n-2j}.\]
Split the left product into three parts:
\[\prod_{j=0}^{n/2}q_{m-2j+1}=
\prod_{j=0}^{\lceil(m-1)/2\rceil}q_{m-2j+1}
\prod_{j=\lceil(m+1)/2\rceil}^{m}q_{m-2j+1}
\prod_{j=m+1}^{n/2}q_{m-2j+1}.\]
Then \eqref{eq:refl} implies that
\[\prod_{j=m+1}^{n/2}q_{m-2j+1}=\prod_{j=m}^{n/2-1}q_{m+n-2j},\]
so it remains to be shown that
\[\prod_{j=0}^{\lceil(m-1)/2\rceil}q_{m-2j+1}
\prod_{j=\lceil(m+1)/2\rceil}^{m}q_{m-2j+1}
=\prod_{j=1}^{m+1}q_j.\]
Applying \eqref{eq:refl} to the second product on the left, we get
\begin{align*}
\prod_{j=0}^{\lceil(m-1)/2\rceil}q_{m-2j+1}
\prod_{j=\lceil(m+1)/2\rceil}^{m}q_{m-2j+1}
&=\prod_{j=0}^{\lceil(m-1)/2\rceil}q_{m-2j+1}
\prod_{j=0}^{\lfloor(m-1)/2\rfloor}q_{m-2j}\\
&=\prod_{j=1}^{m+1}q_j.
\end{align*}
This completes the proof of Theorem 2.

\section{Using Gosper's algorithm}\label{sec:gosper}

Here we give proofs of the identities
\[\pFq{3}{2}{1-\frac{n}{3},\frac{4}{3}-\frac{n}{3},2}
{\frac{1}{6}(11-n+X_n),\frac{1}{6}(11-n-X_n)}{1}
=\frac{n^2-4n+6}{3n-6}\]
and
\[\pFq{3}{2}{1-\frac{n}{2},\frac{3}{2}-\frac{n}{2},2}
{\frac{1}{4}(9+2m-2n+Y_{m\normalcomma n}), \frac{1}{4}(9+2m-2n-Y_{m\normalcomma n})}{1}
=\frac{6+4m-5n+n^2}{4m}\]
that were obtained using Gosper's algorithm.
Since an understanding of how Gosper's algorithm
works is not required for following the proof below,
the reader is referred to \cite{AeqB} for such details.
We can write the first sum as $\sum_{k=0}^{n/3-1} t_k$, where
\[t_k=\frac{(1-\frac{n}{3})_k(\frac{4}{3}-\frac{n}{3})_k(2)_k}
{\bigl(\frac{1}{6}(11-n+X_n)\bigr)_k\bigl(\frac{1}{6}(11-n-X_n)\bigr)_k}\frac{1}{k!}.\]
Gosper's algorithm finds a rational function $R(k)$ such that
$t_k=R(k+1)t_{k+1}-R(k)t_k$, so that the sum telescopes, giving
\[\sum_{k=0}^{n/3-1} t_k
=R(n/3)t_{n/3}-R(0)t_0.\]
Verifying such a proof amounts to computing the rational function
$\frac{1}{t_k}(R(k+1)t_{k+1}-R(k)t_k)$ and checking that it is equal to one, 
then computing and simplifying the expression $R(n/3)t_{n/3}-R(0)t_0$. 
It is routine to do this using computer algebra systems such as Mathematica 
(though it can also be done by hand). 
The author obtained these identities using 
Paule and Schorn's Mathematica package \texttt{fastZeil} \cite{fastZeil};
see \hyperref[sec:mathematica]{Appendix B} for a guide on how
to follow the computations in this paper using Mathematica.

For the first identity, we have
\[R_1(k)=-\frac{6+9k^2-3k(n-5)+n(n-4)}{3(k+1)(n-2)}.\]

Similarly, we can write the second sum as $\sum_{k=0}^{n/2}u_k$,
where
\[u_k=\frac{(1-\frac{n}{2})_k(\frac{3}{2}-\frac{n}{2})_k(2)_k}
{\bigl(\frac{1}{4}(9+2m-2n+Y_{m,n})\bigr)_k\bigl(\frac{1}{4}(9+2m-2n-Y_{m,n})\bigr)_k}\frac{1}{k!};\]
the proof then follows from setting
\[R_2(k)=-\frac{2(1+k)(3+2k+2m)-(5+4k)n+n^2}{4(1+k)m},\]
so that
\[\sum_{k=0}^{n/2}u_k=R_2(n/2+1)u_{n/2+1}-R_2(0)u_0.\]

\section*{Acknowledgements}
The author thanks Hao Huang for his helpful comments on a draft of this paper,
as well as Zhi Tat Ang for numerous stimulating discussions.
The author would also like to express his gratitude to Aryan Jain for his warm 
hospitality during a brief night's stay, where a rejuvenating beverage provided 
invaluable inspiration while working on the proof of Theorem 2.

\newpage
\section*{Appendix A: Technical lemmas \& calculations for similar cases}\label{sec:appendix}
For convenience, we reproduce Sheppard's proof from \cite{sheppard} in our notation.
\renewcommand{\thelemma}{A\arabic{lemma}}
\stepcounter{lemma}
\renewcommand{\theappendixtheorem}{A\arabic{appendixtheorem}}
\begin{appendixtheorem}[Sheppard \cite{sheppard}, Equation (18)]
\label{appendixtheorem:A1}
Suppose $N\ge0$ is a nonnegative integer, and $a,b,d,e\in\CC$. Then
\[\pFq{3}{2}{-N,a,b}{d,e}{1}
=\frac{(d-a)_N(e-a)_N}{(d)_N(e)_N}
\pFq{3}{2}{-N,a,a+b-N-d-e+1}{a-N-d+1,a-N-e+1}{1},\]
whenever both sides converge.
\end{appendixtheorem}

\begin{proof}
By induction on $s$ with the base case $s=1$ given by
\begin{align*}
&1+\frac{(-N)(a)(-a+N+d+e-1)}{de}-\frac{(-a+N+d-1)(-a+N+e-1)}{de}\\
&\quad=\frac{(1-N)(a+1)(-a+N+d+e-1)}{de},
\end{align*}
we have
\begin{align*}
&\sum_{k=0}^s\frac{(-N)_k(a)_k(-a+N+d+e-1)_k}{(d)_k(e)_k\;k!}\\
&\qquad-\frac{(-a+N+d-1)(-a+N+e-1)}{de}\\
&\qquad\quad\times\sum_{k=0}^{s-1}
\frac{(1-N)_k(a+1)_k(-a+N+d+e-1)_k}{(d+1)_k(e+1)_k\;k!}\\
&\quad=\frac{(1-N)_s(a+1)_s(-a+N+d+e-1)_s}{(d)_s(e)_s\;s!}
\end{align*}
for $s\ge1$; applying this identity with $s=N$ gives
\begin{equation} \label{eqn:magic}
\begin{aligned}
&\pFq{3}{2}{-N,a,-a+N+d+e-1}{d,e}{1}\\
&\quad=\frac{(-a+N+d-1)(-a+N+e-1)}{de}\\
&\quad\qquad\times\pFq{3}{2}{1-N,a+1,-a+N+d+e+1}{d+1,e+1}{1}.
\end{aligned}
\end{equation}
Repeated applications of \eqref{eqn:magic} with $s=N,\dots,1$ then yield
\begin{equation} \label{eqn:hill}
\pFq{3}{2}{-N,a,-a+N+d+e-1}{d,e}{1}
=\frac{(d-a)_N(e-a)_N}{(d)_N(e)_N}.
\end{equation}
Now, let
\begin{equation*}
\begin{aligned}
\Phi(N,a,b,d,e)
&\coloneqq(d)_N(e)_N\;\pFq{3}{2}{-N,a,b}{d,e}{1}\\
&=\sum_{k=0}^N\frac{1}{k!}(-N)_k(a)_k(b)_k(d+k)_{N-k}(e+k)_{N-k},
\end{aligned}
\end{equation*}
\begin{equation} 
\begin{aligned} \label{eqn:vanish}
&\Psi(N,a,b,d,e)
\coloneqq(d-a)_N(e-a)_N\;\pFq{3}{2}{-N,a,a+b-N-d-e+1}{a-N-d+1,a-N-e+1}{1}\\
&\qquad=\sum_{k=0}^N\frac{1}{k!}(-N)_k(a)_k(a+b-N-d-e+1)_k(d-a)_{N-k}(e-a)_{N-k},
\end{aligned}
\end{equation}
and
\begin{equation*}
f_N(b)\coloneqq\Phi(N,a,b,d,e)-\Psi(N,a,b,d,e).
\end{equation*}
It suffices to prove that $f_N=0$ for $N\ge0$;
we will induct on $N$. Clearly $f_0=0$, so suppose that
$f_j=0$ for $0\le j<N$. Then
\[\Phi(N,a,b+1,d,e)-\Phi(N,a,b,d,e)
=-Na\,\Phi(N-1,a+1,b+1,d+1,e+1)\]
and
\[\Psi(N,a,b+1,d,e)-\Psi(N,a,b,d,e)
=-Na\,\Psi(N-1,a+1,b+1,d+1,e+1),\]
which implies that $f_N(b)=f_N(b+1)=f_N(b+2)=\cdots$.
Thus $f_N$ is a constant, since it is a polynomial.
Choosing $b=-a+N+d+e-1$, so that $a+b-N-d-e+1=0$,
we see that $\Psi$ reduces to the first term in the sum \eqref{eqn:vanish},
so that $\Psi=(d-a)_N(e-a)_N$ for this choice of $b$.
It now follows from \eqref{eqn:hill} that $\Phi=\Psi$ for this $b$, 
so that $f_N(b)=0$ and consequently $f_N=0$. 
This closes the induction and completes the proof.
\end{proof}

Now we prove some auxiliary results for \hyperref[sec:theorem2]{Section 2}
that ensure that we never have to worry about dividing by zero.
Recall that $Y_{m,n}\coloneqq\sqrt{(2m+1)^2-8mn}$.

\begin{lemma} \label{lemma:A2}
The terms $\frac{1}{4}(9+2m-2n\pm Y_{m,n})$ are never equal to nonpositive integers.
\end{lemma}

\begin{proof}
If $m<2n-1$, then $(2m+1)^2-8mn<0$, so the terms have a nonzero imaginary part.
If $m\ge 2n-1$, then $(2m+1)^2-8mn>0$, and so $Y_{m,n}>0$. Thus it suffices to prove that
$\frac{1}{4}(9+2m-2n-Y_{m,n})\ge2$, or $2(m-n)+1\ge Y_{m,n}$, which reduces to
$4n(n-1)\ge0$.
\end{proof}

The following lemma guarantees that the rising factorials in the denominator of
the expression obtained from applying Sheppard's identity in \hyperref[sec:theorem2]{Section 2}
never contain zero factors.
\begin{lemma} \label{lemma:A3}
The terms $\frac{1}{4}(5-2m-2n\pm Y_{m,n})$ 
and $\frac{1}{4}(-1-2m\pm Y_{m,n})$
are never nonpositive integers in the range
$\{2-\frac{n}{2}, 3-\frac{n}{2}, \dots, 0\}$.
\end{lemma}

\begin{proof}
If $m<2n-1$, all the terms have a nonzero imaginary part.
If $m\ge2n-1$, then $Y_{m,n}>0$. For the terms
$\frac{1}{4}(5-2m-2n\pm Y_{m,n})$, it suffices to prove that
$\frac{1}{4}(5-2m-2n+Y_{m,n})\le\frac{3}{2}-n$, which reduces to
$2(n-m)-1\le-Y_{m,n}$, an inequality established in the proof of \hyperref[lemma:A2]{Lemma A2}.

Similarly, when $m\ge2n-1$, we can prove that $\frac{1}{4}(-1-2m\pm Y_{m,n})\le-\frac{n}{2}$,
since it also reduces to $2(n-m)-1\le-Y_{m,n}$.
\end{proof}

Now we give some details on the cases $n\equiv1,2\pmod3$ for \hyperref[theorem:1]{Theorem 1}
and $n\equiv1\pmod2$ for \hyperref[theorem:2]{Theorem 2}.

When $n\equiv1\pmod3$ in \hyperref[theorem:1]{Theorem 1}, we apply Sheppard's identity with
$N=n/3-4/3$, $a=1-n/3$, $b=5/3-n/3$, $d=\frac{1}{6}(9-3n+X_n)$, $e=\frac{1}{6}(9-3n-X_n)$
to get
\begin{align*}
&\frac{\bigl(\frac{1}{6}(3-n+X_n)\bigr)_{n/3-4/3}
\bigl(\frac{1}{6}(3-n-X_n)\bigr)_{n/3-4/3}}
{\bigl(\frac{1}{6}(9-3n+X_n)\bigr)_{n/3-4/3}
\bigl(\frac{1}{6}(9-3n-X_n)\bigr)_{n/3-4/3}}\\
&\quad\times\pFq{3}{2}{\frac{4}{3}-\frac{n}{3}, 1-\frac{n}{3}, 2}
{\frac{1}{6}(11-n+X_n), \frac{1}{6}(11-n-X_n)}{1}\\
&\quad=\frac{\prod_{j=0}^{n/3-7/3}\bigl(j^2+\frac{3-n}{3}j+\frac{2-3n+n^2}{9}\bigr)}
{\prod_{j=0}^{n/3-7/3}\bigl(j^2+(3-n)j+\frac{3n^2-15n+20}{9}\bigr)}
\frac{n^2-4n+6}{3n-6}.
\end{align*}
(Notice that we obtain the exact same ${}_3F_2(1)$ sum as in \hyperref[sec:theorem1]{Section 1},
since the ordering of parameters does not change a hypergeometric series.) From here the task is
then to prove that this expression is equal to
\[\frac{\lfloor n/3\rfloor}{\prod_{j=1}^{\lfloor n/3\rfloor-1}c_{3j}}
\frac{\prod_{j=n+1}^{\lfloor3n/2\rfloor-2}c_j}{\prod_{j=\lfloor n/3\rfloor+1}^{\lfloor n/2\rfloor-1}c_{3j}},\]
which proceeds in essentially the same way as before, except now things are slightly messier because of the
additional floor functions.

When $n\equiv2\pmod3$ in \hyperref[theorem:1]{Theorem 1}, we apply Sheppard's identity with
$N=n/3-5/3$, $a=1-n/3$, $b=4/3-n/3$, $d=\frac{1}{6}(9-3n+X_n)$, $e=\frac{1}{6}(9-3n-X_n)$
to obtain
\begin{align*}
&\frac{\bigl(\frac{1}{6}(3-n+X_n)\bigr)_{n/3-5/3}
\bigl(\frac{1}{6}(3-n-X_n)\bigr)_{n/3-5/3}}
{\bigl(\frac{1}{6}(9-3n+X_n)\bigr)_{n/3-5/3}
\bigl(\frac{1}{6}(9-3n-X_n)\bigr)_{n/3-5/3}}\\
&\quad\times\pFq{3}{2}{\frac{5}{3}-\frac{n}{3}, 1-\frac{n}{3}, 2}
{\frac{1}{6}(13-n+X_n), \frac{1}{6}(13-n-X_n)}{1}\\
&\quad=\frac{\prod_{j=0}^{n/3-8/3}\bigl(j^2+\frac{3-n}{3}j+\frac{n^2-3n+2}{9}\bigr)}
{\prod_{j=0}^{n/3-8/3}\bigl(j^2+(3-n)j+\frac{3n^2-15n+20}{9}\bigr)}
\frac{n^2-5n+12}{3n-3}.
\end{align*}
Notice that the fraction with two products on the left is identical to the one that
arises in the $n\equiv1\pmod3$ case, but the hypergeometric series is no longer the same!
Nevertheless, the sum succumbs to Gosper's algorithm, with
\[R(k)=-\frac{12+9k^2-3k(n-7)+n(n-5)}{3(1+k)(n-1)};\]
the same manipulations then follow, except we now have
\[\frac{b_0}{c_{n+2}}
=\frac{\lfloor n/3\rfloor(3n-3)}{n^2-5n+12}
=\frac{(n-2)(n-1)}{n^2-5n+12}.\]

Finally, when $n\equiv1\pmod2$ in \hyperref[theorem:2]{Theorem 2}, we apply 
Sheppard's identity with $N=n/2-3/2$, $a=1-n/2$, $b=1-m$, 
$d=\frac{1}{4}(9+2m-2n+Y_{m,n})$, and $e=\frac{1}{4}(9+2m-2n-Y_{m,n})$
to get
\begin{align*}
&\frac{\bigl(\frac{1}{4}(1-2m+Y_{m,n})\bigr)_{n/2-3/2}
\bigl(\frac{1}{4}(1-2m-Y_{m,n})\bigr)_{n/2-3/2}}
{\bigl(\frac{1}{4}(5-2m-2n+Y_{m,n})\bigr)_{n/2-3/2}
\bigl(\frac{1}{4}(5-2m-2n-Y_{m,n})\bigr)_{n/2-3/2}}\\
&\qquad\times\pFq{3}{2}{\frac{3}{2}-\frac{n}{2}, 1-\frac{n}{2}, 2}
{\frac{1}{4}(9+2m-2n+Y_{m\normalcomma n}), \frac{1}{4}(9+2m-2n-Y_{m\normalcomma n})}{1}\\
&\quad=\frac{\prod_{j=0}^{n/2-5/2}\bigl(j^2-\frac{2m-1}{2}j+\frac{m(n-1)}{2}\bigr)}
{\prod_{j=0}^{n/2-5/2}\bigl(j^2+\frac{5-2m-2n}{2}j+\frac{6-6m-5n+4mn+n^2}{4}\bigr)}
\frac{6+4m-5n+n^2}{4m}.
\end{align*}
The resulting hypergeometric series is identical to that given in \hyperref[sec:theorem2]{Section 2},
and the products are almost identical, except that the $mn/2$ in the original numerator is now
$m(n-1)/2$. The arguments that follow are also essentially the same.

\newpage
\section*{Appendix B: Using Mathematica to follow this paper}\label{sec:mathematica}
Most of the computations in this paper were done and verified using Wolfram Mathematica.
The computations involving Gosper's algorithm were done using Paule and Schorn's 
Mathematica package \texttt{fastZeil} \cite{fastZeil}. For example, the first few
lines of \hyperref[sec:theorem1]{Section 1} written in Mathematica would look something like this:
\\
\begin{adjustwidth}{-75pt}{100pt}
\begin{mmaCell}[moredefined={X},morepattern={n_, n}]{Input}
  X[n_] := Sqrt[1+6n-3\mmaSup{n}{2}]

\end{mmaCell}

\begin{mmaCell}{Input}
  -Binomial[n-3j,3] == -\mmaFrac{1}{6}(n-3j)(n-1-3j)(n-2-3j)

\end{mmaCell}

\begin{mmaCell}{Output}
  True

\end{mmaCell}

\begin{mmaCell}[moredefined={X}]{Input}
  FullSimplify[Binomial[n,3]-Binomial[n-3j,3]

    == \mmaFrac{9}{2}j(j-\mmaFrac{1}{6}(3n-3+X[n]))(j-\mmaFrac{1}{6}(3n-3-X[n]))]

\end{mmaCell}

\begin{mmaCell}{Output}
  True

\end{mmaCell}
\end{adjustwidth}

And the proof of the first identity in \hyperref[sec:gosper]{Section 3} would look like this: 
\\
\begin{adjustwidth}{-75pt}{100pt}
\begin{mmaCell}{Input}
  << RISC`fastZeil` (* import fastZeil *)

\end{mmaCell}

\begin{mmaCell}[moredefined={Gosper, X}]{Input}
  Gosper[\mmaFrac{Pochhammer[1-\mmaFrac{n}{3},k]Pochhammer[\mmaFrac{4}{3}-\mmaFrac{n}{3},k]Pochhammer[2,k]}{Pochhammer[\mmaFrac{1}{6}(11-n+X[n]),k]Pochhammer[\mmaFrac{1}{6}(11-n-X[n]),k]k!}, \{k,0,\mmaFrac{n}{3}-1\}]

\end{mmaCell}

\begin{mmaCell}{Print}
  If -1+n/3 is a natural number and
     (-3+n)/3 is no negative integer and
     -2+n != 0, then: 

\end{mmaCell}

\begin{mmaCell}{Output}
  \{Sum[\mmaFrac{Pochhammer[2,k] Pochhammer[1-\mmaFrac{n}{3},k] Pochhammer[\mmaFrac{4}{3}-\mmaFrac{n}{3},k]}{k!Pochhammer[\mmaFrac{1}{6}(11-n-\mmaSqrt{1+6n-3\mmaSup{n}{2}}),k] Pochhammer[\mmaFrac{1}{6}(11-n+\mmaSqrt{1+6n-3\mmaSup{n}{2}}),k]},\{k,0,-1+\mmaFrac{n}{3}\}]

    == -\mmaFrac{(5-n+\mmaSqrt{1+6n-3\mmaSup{n}{2}}) (-5+n+\mmaSqrt{1+6n-3\mmaSup{n}{2}})}{12(-2+n)}\}

\end{mmaCell}

\begin{mmaCell}{Input}
  Prove[]

\end{mmaCell}

\begin{mmaCell}{Input}
  FullSimplify[-\mmaFrac{(5-n+\mmaSqrt{1+6n-3\mmaSup{n}{2}}) (-5+n+\mmaSqrt{1+6n-3\mmaSup{n}{2}})}{12(-2+n)}]

\end{mmaCell}

\begin{mmaCell}{Output}
  \mmaFrac{6+(-4+n)n}{3(-2+n)}

\end{mmaCell}
\end{adjustwidth}

Here, the \texttt{Prove[]} function from the \texttt{fastZeil} package generates a proof of the identity
in a separate Mathematica notebook, which amounts to giving the rational function $R(k)$ as discussed
in \hyperref[sec:gosper]{Section 3}.

If the reader owns a copy of Mathematica and wishes to follow the computations in this paper,
a Mathematica notebook containing the essential computations of this paper is available from 
\url{https://boonsuan.github.io/orderings.nb}.
\end{document}